\documentclass[10pt]{amsart}
\usepackage[utf8]{inputenc}
\usepackage{color}
\usepackage{bbm}
\usepackage{amsmath,amssymb}
\usepackage{tikz}
\usetikzlibrary{arrows}
\usepackage{enumerate}
\usepackage{hyperref}

\newtheorem{thm}{Theorem}[section]
\newtheorem{cor}[thm]{Corollary}
\newtheorem{lem}[thm]{Lemma}
\newtheorem{prop}[thm]{Proposition}
\theoremstyle{definition}
\newtheorem{defin}[thm]{Definition}
\newtheorem{rem}[thm]{Remark}

\numberwithin{equation}{section}

\newcommand{\bN}{\mathbb{N}}
\newcommand{\bR}{\mathbb{R}}

\newcommand{\supp}{\operatorname{supp}}

\newcommand{\dif}{\,\mathrm{d}}

\newcommand{\cA}{\mathcal{A}}

\newcommand{\bE}{\mathbb{E}}

\newcommand{\charfun}{\ensuremath{\mathbbm 1}}

\DeclareMathOperator{\co}{conv}

\allowdisplaybreaks

\begin{document}
\title{Almost everywhere Convergence  of Spline Sequences}
\author[P. F. X. Müller]{Paul F. X. Müller}
\address{Institute of Analysis, Johannes Kepler University Linz, Austria, 4040 Linz, Altenberger Strasse 69}
\email{paul.mueller@jku.at}

\author[M. Passenbrunner]{Markus Passenbrunner}
\address{Institute of Analysis, Johannes Kepler University Linz, Austria, 4040 Linz, Altenberger Strasse 69}
\email{markus.passenbrunner@jku.at}

\date{\today}
\begin{abstract}
We prove the analogue of  the Martingale Convergence Theorem for polynomial spline sequences. 
Given a natural number $k $ and a sequence $(t_i)$ of knots in $[0,1]$ with multiplicity 
$\le k-1$, we let $P_n $ be the orthogonal projection onto 
the space of spline polynomials in $[0,1] $ of degree $k-1$ corresponding to the grid
$(t_i)_{i=1}^n$. Let $X$ be a Banach space with the 
Radon-Nikod\'{y}m property. Let  $(g_n)$ be a bounded sequence in the 
Bochner-Lebesgue space $L^1_X [0,1]$ satisfying 
$$ g_n = P_n ( g_{n+1} ),\qquad n \in \bN . $$
We prove the  existence of $\lim_{n\to \infty} g_n(t) $ in $X$  for almost every $t \in [0,1]. $ 
Already in the scalar valued case $X = \bR $ the  result is new.      
\end{abstract}
\maketitle
\section{Introduction}
In this paper we prove 
a convergence theorem for splines in vector valued $L^1 $-spaces. 
By way of introduction we consider the analogous convergence theorems for  martingales
with respect to a filtered probability space $ (\Omega , (\cA_n ) , \mu ) . $  We first review two  classical theorems for scalar valued martingales in 
$L^1 = L^1 (\Omega , \mu ) $. See Neveu \cite{Neveu1975}.
\begin{itemize}
\item[(M1)] Let $ g \in L^1 . $ If $g_n = \bE( g | \cA_n ) $ then $\|g_n \|_1  \le  \|g \|_1 $ and $ ( g_n ) $ 
converges almost everywhere and in $L^1 . $ 
\item[(M2)] Let   $ ( g_n ) $  be a bounded sequence in $L^1 $ such that $   g_n = \bE( g_{n+1} | \cA_n ) $. Then 
 $ ( g_n ) $ converges almost everywhere and $g = \lim g_n $ satisfies $ \|g \|_1  \le  \sup \|g_n  \|_1 $. 
\end{itemize}
Next we turn to  vector valued martingales. We  fix a Banach space $X$ and let
$L^1_X = L^1_X (\Omega , \mu ) $ denote the Bochner-Lebesgue space.  
The Radon-Nikod\'{y}m property (RNP) of the Banach space $X$ is intimately 
tied to martingales in Banach spaces. We refer to the book by  Diestel and Uhl \cite{DiestelUhl1977} for the following 
basic and well known results. 
\begin{itemize}
\item[(M3)] Let $ g \in L^1_X . $ If $g_n = \bE( g | \cA_n ) $ then $\|g_n \|_{L^1_X}  \le  
\|g \|_{L^1_X} $. The sequence  $ ( g_n ) $ 
converges almost everywhere in $X$ and in $L^1_X . $ (This holds  for any Banach space 
 $X$.) 
\item[(M4)] Let   $ ( g_n ) $  be a bounded sequence in $L^1_X $ such that 
$   g_n = \bE( g_{n+1} | \cA_n ) $. 
If the Banach space $X$ satisfies the Radon-Nikod\'{y}m property, 
then 
 $ ( g_n ) $ converges almost everywhere in $X$  and $g = \lim g_n $ satisfies 
$ \|g \|_{L^1_X}  \le  \sup \|g_n  \|_{L^1_X} $. 
Moreover the $L^1_X$-density of the
 $\mu$-absolutely 
continuous part of the vector measure 
$$ \nu (E) = \lim_{n\to \infty} \int_E g_n  d\mu, \qquad E \in \cup \cA_n $$ 
determines  $g = \lim g_n $. 
\item[(M5)] Conversely if $X$ fails to satisfy the Radon Nikod\'{y}m property, then there 
exists     a filtered probability space $ (\Omega , (\cA_n ) , \mu )  $ 
and bounded sequence in $L^1_X (\Omega , \mu ) $ satisfying $  g_n = \bE( g_{n+1} | \cA_n ) $
such that $(g_n)$ fails to converge almost everywhere in $X$. 
\end{itemize}
In the present paper we establish a new link between 
probability (almost sure convergence of martingales, the RNP) and 
approximation theory (projections onto splines in $[0,1]$).

We review the  basic setting pertaining to  spline projections. 
(See for instance \cite{Shadrin2001}, \cite{PassenbrunnerShadrin2014}, \cite{Schumaker2007}.)
So, fix an integer $ k \geq 2 $,  let $(t_i)$ a sequence of grid points in  $(0,1 )$
 where each $t_i $  occurs at
most $k-1$ times. We emphasize that in contrast to \cite{PassenbrunnerShadrin2014},
in the present paper we don't assume that the sequence of grid points
is dense in  $(0,1 )$.

Let $S_n $ denote the space of splines on the interval $[0,1] $ of order $k$ 
(degree $k-1$) corresponding to the grid
$(t_i)_{i=1}^n$. 
Let  $\lambda$ denote Lebesgue measure on the unit interval $[0,1].$
Let $P_n$ be the orthogonal projection with
respect to $L^2([0,1], \lambda)$ onto the
space of splines $S_n$.
 By Shadrin's theorem \cite{Shadrin2001},
$P_n $ admits an extension to    $L^1([0,1], \lambda)$ such that 
$$ \sup_{n \in \bN } \| P_n : L^1([0,1], \lambda) \to L^1([0,1], \lambda) \| <\infty .$$ 
Assuming that the  sequence  $(t_i)$ 
is  dense in  the unit interval $[0,1]$,
the second named author and A. Shadrin \cite{PassenbrunnerShadrin2014} proved -- in effect -- that 
for any 
$g \in L^1_X ( [0,1], \lambda)$ the sequence $ g_n = P_n g $  converges almost everywhere in $X$.
The vector valued version  of \cite{PassenbrunnerShadrin2014} holds true  without 
any condition on the underlying Banach space $X$. 
Thus the paper  \cite{PassenbrunnerShadrin2014} established the spline analogue of the 
martingale  properties (M1) and (M3) -- under the restriction that  $(t_i)$ 
is  dense in  the unit interval $[0,1]$. 

Our main theorem -- extending  \cite{PassenbrunnerShadrin2014} -- 
shows that  the  vector valued martingale convergence theorem 
has a direct counterpart 
in  the context of spline projections. Theorem \ref{thm:main} gives the spline 
analogue of the martingale  properties (M2) and (M4). The first step in the proof of 
Theorem \ref{thm:main} consists in showing that 
the restrictive density condition  on   $(t_i)$ may be lifted from the assumptions in 
 \cite{PassenbrunnerShadrin2014}. 
\begin{thm}[Spline Convergence Theorem]\label{thm:main}
	Let $X$ be a Banach space with RNP and 
	$(g_n)$ be a sequence in $L^1_X$ with the properties
	\begin{enumerate}
		\item $\sup_n \|g_n\|_{L^1_X} <\infty$, 
		\item $P_m g_n = g_m$ for all $m\leq n$.
	\end{enumerate}
	Then, $g_n$ converges $\lambda$-a.e. to some $L^1_X$ function.
\end{thm}
\noindent 
Already in the scalar case $X = \bR $ Theorem~\ref{thm:main} is a new 
result.
In the course of  its proof 
we  {\em intrinsically} describe  the pointwise limit of the sequence  $( g_n )$. 
At the end   of Section \ref{sec:proof} 
we formulate a refined version of  Theorem~\ref{thm:main} 
employing  the tools we developed 
for its proof.  This includes an  explicit expression of 
$\lim g_n $ in terms of $B$-splines. 
   
We point out that only under significant restrictions on the geometry  of the  grid points   $(t_i)$, 
is it true that the spline projections $P_n$ are Calderon-Zygmund operators (with constants independent of  $n $).
See \cite{gev-kam-08}.

Our present paper should be seen in context with the second named author's work \cite{Passenbrunner2014}, where Burkholder's martingale inequality
$$ \big\| \sum \pm ( \bE( f_{} | \cA_n ) - \bE( f_{} | \cA_{n-1} )) \big \|_{L^p (\Omega, \mu )} \le C_p 
\| f\|  _{L^p (\Omega ,\mu)} , $$
was given a counterpiece for spline projections as follows
$$ \big \| \sum \pm ( P_n(g)  -  P_{n-1}(g) )\big \|_{L^p ([0,1] ) }\le C_p 
\| g \|  _{L^p ([0,1] )},  $$ 
where $ 1 < p < \infty, $ and $C_p \sim p^2/(p-1) . $ 
The corresponding analogue for vector valued spline projections is still outstanding.
(See however \cite{KamontMuller2006} for  a special case.)

\subsection*{Organization.} The  presentation is organized as follows. 
In Section \ref{sec:prelims}, we
collect some important facts and tools used in this article. Section \ref{sec:convFunction} treats the convergence of $P_n g$ for
$L^1_X$-functions $g$. 
Section \ref{sec:splineConstruction} contains special spline constructions associated to
the point sequence $(t_i)$.
In Section \ref{sec:measure}, we give a measure theoretic lemma
that is subsequently employed and may be of independent interest in the theory of
splines.
Finally, in Section \ref{sec:proof}, we give
the proof of the Spline Convergence Theorem.

\section{Preliminaries}
\label{sec:prelims}
\subsection{Basics about vector measures}
We refer to the book \cite{DiestelUhl1977} by J. Diestel and J.J. Uhl
for basic facts on martingales and vector measures. 
Let $(\Omega,\mathcal A)$ be a measure space and $X$ a Banach space.
Every $\sigma$-additive map $\nu:\mathcal A\to X$ is called a \emph{vector
measure}.
The \emph{variation} $|\nu|$ of $\nu$ is the set function
\begin{equation*}
	|\nu|(E) = \sup_\pi \sum_{A\in\pi} \|\nu(A)\|_X,
\end{equation*}
where the supremum is taken over all partitions $\pi$ of $E$ into a finite numer
of pairwise disjoint members of $\mathcal A$.
If $\nu$ is of bounded variation, i.e., $|\nu|(\Omega)<\infty$, the variation
$|\nu|$ is $\sigma$-additive.
If $\mu:\mathcal A \to [0,\infty)$ is a measure and $\nu:\mathcal A\to X$ is
	a vector measure, $\nu$ is called \emph{$\mu$-continuous}, if
	$\lim_{\mu(E)\to 0} \nu(E)=0$ for all $E\in\mathcal A$.

	\begin{defin}
		A Banach space $X$ has the \emph{Radon-Nikod\'{y}m property
	(RNP)} if
		for every measure space $(\Omega,\mathcal A)$, for every
		positive measure $\mu$ on $(\Omega,\mathcal A)$ and for every
		$\mu$-continuous vector measure $\nu$ of bounded variation,
		there exists a function $f\in L^1_X(\Omega,\mathcal A,\mu)$ such
		that
		\begin{equation*}
			\nu(A) = \int_A f\dif\mu,\qquad A\in\mathcal A.
		\end{equation*}
	\end{defin}

\begin{thm}[Lebesgue decomposition of vector measures]\label{thm:lebesgue}
Let $(\Omega,\mathcal A)$ be a measure space, $X$ a Banach space,
$\nu : \mathcal A \to X$ a vector measure and $\mu:\mathcal A\to [0,\infty)$
	a measure. Then, there exist unique vector
measures $\nu_c,\nu_s : \mathcal A\to X$ such that
\begin{enumerate}
	\item $\nu = \nu_c + \nu_s$,
	\item $\nu_c$ is $\mu$-continuous,  
	\item $x^*\nu_s$ and $\mu$ are mutually singular for each $x^*\in
		X^*$.
\end{enumerate}
If $\nu$ is of bounded variation, $\nu_c$ and $\nu_s$ are of bounded
variation as well, $|\nu|(E) = |\nu_c|(E) + |\nu_s|(E)$ for each $E\in \mathcal A$ and
$|\nu_s|$ and $\mu$ are mutually singular.
\end{thm}
The following theorem provides the fundamental  
link between convergence of vector valued martingales
and the RNP 
of the underlying Banach space $X$. See Diestel-Uhl \cite[Theorem V.2.9]{DiestelUhl1977}.  It is the point of reference for our present work 
on convergence of spline projections.   
\begin{thm}[Martingale convergence theorem]\label{thm:mart_conf}
Let $(\Omega,\mathcal A)$ be a  measure space and $\mu:\mathcal A\to [0,\infty)$
a measure. 
Let $(\cA_n)$ be a sequence of increasing sub-$\sigma$-algebras of  $\cA .$ 
Let $X$ be a Banach space, let $(g_n)$ be a bounded sequence in 
$L^1_X (\Omega,\mathcal A_n, \mu ), $ such that $g_n = \bE( g_{n+1} |\cA )$ 
and let 
$$\nu ( E ) = \lim_{n\to \infty} \int_E g_n d\mu, \qquad  E \in \cup \cA_n .$$ 
Let $\nu = \nu_c +\nu_s $ denote the Lebesgue decomposition of $\nu $
with respect to $ \mu .$ Then $\lim_{n\to \infty} g_n$ exists almost everywhere with respect to $\mu $ 
if and only if $\nu_c$ has a Radon-Nikod\'{y}m derivative $f\in L^1_X(\Omega, \mu ). $ 
In this case 
$$ \lim_{n\to \infty} g_n = \bE( f | \cA_\infty) , $$ 
where $   \cA_\infty$ is the $\sigma$-algebra generated by $\cup \cA_n . $ 
\end{thm}   
Let $ X $ be a Banach space, let  $ v \in   L^1(\Omega,\mathcal A,m)$ and $ x \in X . $
We recall that  $ v\otimes x : \Omega \to X $ is defined by  $v\otimes x (\omega ) = v(\omega) x$ 
and that 
$$ L^1(\Omega,\mathcal A,m) \otimes X =  {\rm span} \{ v_i  \otimes x_i : v_i \in L^1 (\Omega,\mathcal A,m) , x_i \in X \} . $$

The following  lemmata are taken from 
\cite{Pisier2016}.
\begin{lem} 
For any Banach space  $X$, the algebraic tensor product  
$ L^1(\Omega,\mathcal A,m) \otimes X   $ is a dense subspace of the Bochner-Lebesgue space  
$L^1_X (\Omega,\mathcal A,m) . $
\end{lem}

\begin{lem}\label{lem:extension}
	Given a bounded operator $T : L^1(\Omega,\mathcal A,m) \to
        L^1(\Omega',\mathcal A',m')$ there exists a unique bounded linear map
	$\widetilde{T}:L^1_X(\Omega,\mathcal A,m)\to L^1_X(\Omega',\mathcal
	A',m')$ such that 
	\begin{equation*}
		\widetilde{T}(\varphi\otimes x) = T(\varphi)x,\qquad \varphi\in
		L^1(\Omega,\mathcal A,m),x\in X.
	\end{equation*}
	Moreover, $\|\widetilde{T}\| = \|T\|$.
\end{lem}

\begin{lem}\label{lem:dual}
	Let $X_0$ be a separable closed subspace of a Banach space $X$. Then,
	there exists a sequence $(x_n^*)$ in the unit ball of the dual $X^*$ of
	$X$ such that
	\begin{equation*}
		\| x \| = \sup_n |x_n^*(x)|,\qquad x\in X_0.
	\end{equation*}
\end{lem}

\subsection{Tools from Real Analysis}
We use  the book by E. Stein \cite{Stein1970} 
as our basic reference to Vitali's covering Lemma and  
weak-type estimates for the 
Hardy-Littlewood maximal function.

\begin{lem}[Vitali covering lemma]\label{lem:vitali}
	Let $\{C_x : x\in \Lambda\}$ be an arbitrary collection of balls in
	$\mathbb R ^d$ such that $\sup\{\operatorname{diam}(C_x) : x\in \Lambda\}<\infty$.
	Then, there exists a countable subcollection $\{C_x : x \in J\}$,
	$J\subset \Lambda$ of balls from the original collection that are
	disjoint and satisfy
	\begin{equation*}
		\bigcup_{x\in \Lambda} C_x \subset \bigcup_{x\in J} 5C_x.
	\end{equation*}
\end{lem}
 Vitali's covering Lemma implies 
weak type estimates for the Hardy-Littlewood maximal function. 
\begin{thm}\label{thm:weaktype}
	Let $f\in L^1_X$ and $\mathcal Mf(t) := \sup_{I\ni t}
	\frac{1}{\lambda(I)}\int_I \|f(s)\|_X\dif s$ the Hardy-Littlewood
	maximal function.
	Then $\mathcal M$ satisfies the weak type estimate
	\begin{equation*}
		\lambda( \{ \mathcal Mf > u \} ) \leq 
		\frac{C\|f\|_{L^1_X}}{u},\qquad u>0,
	\end{equation*}
	where $C>0$ is an absolute constant.
\end{thm}

\subsection{Spline spaces}
Denote by
$|\Delta_n|$ the maximal mesh width of the grid $\Delta_n=(t_i)_{i=1}^n$
augmented with $k$ times the boundary points $\{0,1\}$.
Recall that $P_n$ is the orthogonal projection operator onto the space $S_n$ of
splines corresponding to the grid $\Delta_n$, which is a conditional expectation
operator for $k=1$.

For the following, we introduce the notation $A(t)\lesssim B(t)$ to indicate the
existence of a constant $c>0$ that only depends on $k$ such that $A(t)\leq
cB(t)$, where $t$ denotes all explicit or implicit dependences that the
expressions $A$ and $B$ might have.
 As is shown by A. Shadrin,
 the sequence  $(P_n) $ satisfies     $L^1$ estimates  as follows:
\begin{thm}[\cite{Shadrin2001}]\label{thm:shadrin}
The orthogonal projection $P_n$  admits a bounded extension to  $L^1$ 
such that  
$$\sup_n \| P_n : L^1 \to L^1\| \lesssim 1 .  $$ 
\end{thm}

By Lemma
\ref{lem:extension}, the operator $P_n$ can be extended to the
vector valued $L^1$ space $L^1_X$ with the same norm
so that for all $\varphi\in L^1$ and $x\in X$, we have
$P_n(\varphi\otimes x) = (P_n\varphi)x$. We also have the identity
\begin{equation}\label{eq:projidentity}
	\int_0^1 P_n g(t)\cdot f(t)\dif\lambda(t) = \int_0^1 g(t) \cdot P_nf(t)
	\dif\lambda(t),\qquad g\in L^1_X, f\in
L^\infty,
\end{equation}
which is just the extension of the fact
that $P_n$ is selfadjoint on $L^2$.

Fix  $f\in C[0,1]$. Consider  the $k$th forward differences of $f$ given by
$$ D_h^kf(t) = \sum_{j=0}^k(-1)^{k-j}\binom{k}{j}f(t+jh) .$$
The $k$th modulus of smoothness of $f$ in $L^\infty$ is defined as 
$$  \omega_k(f,\delta) =\sup_{0\le h\le \delta}\sup_{ 0\le t\le 1-kh} |D_h^kf(t)| ,$$
where $0\le \delta \le 1/k.$ We have  $\lim_{\delta \to 0}\omega_k(f,\delta) =0$ 
for any  $f\in C[0,1]$. Any continuous function  can be approximated by spline functions satisfying 
the following quantitative error estimate. 
\begin{thm}[{\cite[Theorem 6.27]{Schumaker2007}}]\label{thm:jackson}
	Let $f\in C[0,1]$. Then, 
	\begin{equation*}
		d(f, S_n)_\infty \lesssim \omega_k(f,|\Delta_n|),
	\end{equation*}
	where
	$d(f,S_n)_\infty$ is the distance between $f$ and $S_n$ in the
	$\sup$-norm.
	Therefore, if $|\Delta_n|\to 0$, we have $d(f,S_n)_\infty\to 0$.
\end{thm}

Denote by $(N_i^{(n)})_i$ the B-spline basis of $S_n$ normalized such that it
forms a partition of unity and by $(N_i^{(n)*})_i$ its
corresponding dual basis in $S_n$.  Observe that 
\begin{equation*}
	P_n f(t) = \sum_i \langle f, N_i^{(n)}\rangle N_i^{(n)*}(t),\qquad f\in
	L^2.
\end{equation*}
Since the B-spline functions $N_i^{(n)}$ are contained in $C[0,1],$ we can
also insert $L^1$-functions as well as measures in the above formula.

If we set $a_{ij}^{(n)} = \langle
N_i^{(n)*},N_j^{(n)*}\rangle$, we can expand the dual B-spline functions as a
linear combination of B-spline functions with those coefficients:
\begin{equation}\label{eq:formuladual}
	N_i^{(n)*} = \sum_j a_{ij}^{(n)} N_j^{(n)}.
\end{equation}
Moreover,
for $t\in [0,1]$ denote by $I_n(t)$ a smallest grid point interval of positive
length in the grid $\Delta_n$
that contains the point $t$.  We denote by  $i_n(t)$ the largest index $i$ such that
$I_n(t)\subset \supp N_i^{(n)}$. 
Additionally, denote by $h_{ij}^{(n)}$ the length of the convex
hull of the union of the supports of $N_{i}^{(n)}$ and $N_{j}^{(n)}$. 

With this notation, we can give the following estimate for the numbers $a_{ij}^{(n)}$ and, a
fortiori, for $N_i^{(n)*}$:
\begin{thm}[\cite{PassenbrunnerShadrin2014}]\label{thm:boundgram}
There exists $ q \in ( 0,1) $ depending only on the spline order $k $, such that the numbers $a_{ij}^{(n)} = \langle N_i^{(n)*}, N_j^{(n)*}\rangle$
	satisfy the inequality
	\[
		|a_{ij}^{(n)}| \lesssim \frac{q^{|i-j|}}{h_{ij}^{(n)}},
	\]
	and therefore, in particular,  for all $i$,
	\[
		|N_i^{(n)*}(t)| \lesssim
		\frac{q^{|i-i_n(t)|}}{\max\big(\lambda(I_n(t)),\lambda(\supp
			N_i^{(n)})\big)},\qquad t\in [0,1].
	\]
\end{thm}
\begin{proof}
	The first inequality is proved in \cite{PassenbrunnerShadrin2014} and the second one is an easy
	consequence of the first one inserted in  formula \eqref{eq:formuladual} for
	$N_i^{(n)*}$.
\end{proof}

An almost immediate consequence of this estimate is the following pointwise maximal
inequality for $P_n g$:
\begin{thm}[\cite{PassenbrunnerShadrin2014}]\label{thm:maximal}
	For all $g\in L^1_X$,
\[
	\sup_n \|P_n g(t)\|_X \lesssim \mathcal M g(t),\qquad t\in[0,1],
\]
where $\mathcal M g(t) = \sup_{I\ni t} \frac{1}{\lambda(I)}\int_I \|g(s)\|_X\dif s$ denotes
the Hardy-Littlewood maximal function.
\end{thm}
This result and Theorem \ref{thm:jackson}, 
 combined with Theorem \ref{thm:weaktype},
imply the a.e. convergence of $P_ng$ to $g$ for any
$L^1$-function $g$ provided that the point sequence $(t_i)$ is dense in the unit
interval $[0,1]$, cf. \cite{PassenbrunnerShadrin2014}.

As the spline spaces $S_n$ form an increasing sequence of subspaces of $L^2$, we
can write the B-spline function $N_i^{(n)}$ as a linear combination of the finer
B-spline functions $(N_j^{(n+1)})$. The exact form of this expansion is given by
Böhm's algorithm \cite{Boehm1980} and it states in particular that the following
result is valid:

\begin{prop}\label{thm:boehm}
	Let $f = \sum_i \alpha_i N_i^{(m)}\in S_m$ for some $m$. Then, there
	exists a sequence $(\beta_i)$ of coefficients so that
	\begin{equation*}
		f \equiv \sum_i \beta_i N_i^{(m+1)}
	\end{equation*}
	and, for all $i$, $\beta_i$ is a convex combination of $\alpha_{i-1}$
	and $\alpha_i$.
\end{prop}
By induction, an immediate consequence of this result is  
\begin{cor}\label{cor:boehm}
	For any positive integers $n\geq m$ and any index $i$,
	the B-spline function $N_i^{(m)}$ can be represented as
	\begin{equation*}
		N_i^{(m)} \equiv \sum_j \lambda_j N_j^{(n)},
	\end{equation*}
	with coefficients $\lambda_j\in [0,1]$ for all $j$.
\end{cor}

In the following theorem it is convenient to display explicitly the order $k$ of the B-splines 
$N_i^{(n)}=N_{i,k}^{(n)}$. The relation between the sequences $(N_{i,k}^{(n)})_i$ and  $(N_{i,k-1}^{(n)})_i$ is given by  well known recursion formulae, for which we refer  \cite{DeVoreLorentz1993}. 
See also \cite{Schumaker2007}.
\begin{thm}\label{thm:recursion}
Let $[a,b]=\supp N_{i,k}^{(n)}$. Then, the B-spline function  $N_{i,k}^{(n)}$ of order $k$ can be expressed
in terms of two B-spline functions of order $k-1$ as follows:
\begin{equation*}
N_{i,k}^{(n)}(t) = \frac{t-a}{\lambda(\supp N_{i,k-1}^{(n)})} N_{i,k-1}^{(n)}(t) + \frac{b-t}{\lambda(\supp N_{i+1,k-1}^{(n)})} N_{i+1,k-1}^{(n)}(t).
\end{equation*}
\end{thm}

\section{Convergence of $P_n g$}\label{sec:convFunction}
As we are considering arbitrary sequences of grid points $(t_i)$ which are not
necessarily dense in $[0,1]$, as a first stage in the proof of the Spline
Convergence Theorem, we examine the convergence of $P_n g$ for $g\in L^1_X$.

We first notice that $P_n g$ converges in $L^1$. Indeed, this is a consequence of
the uniform boundedness of $P_n$ on $L^1$ as we will
now show.
Observe that for $g\in L^2$, we get that if we define $S_\infty$ as the $L^2$ closure of $\cup S_n$ 
and $P_\infty$ as the orthogonal projection onto $S_\infty$,
\begin{equation*}
	\| P_n g - P_\infty g\|_{L^2} \to 0.
\end{equation*}
Next, we show that this definition of $P_\infty$ can be extended to $L^1$
functions $g$. So, let $g\in L^1$ and $\varepsilon>0$. Since $L^2$ is dense in
$L^1$, we can choose $f\in L^2$ with the property $\|g-f\|_1<\varepsilon$.
Now, choose $N_0$ sufficiently large that for all $m,n>N_0$, we have
$\|(P_n-P_m)f\|_2<\varepsilon$. Then, we obtain
\begin{align*}
	\| (P_n - P_m) g \|_{L^1} &\leq \|(P_n-P_m)(g-f)\|_{L^1} +
	\|(P_n-P_m)f\|_{L^1} \\
	&\leq 2C\varepsilon + \|(P_n-P_m)f\|_{L^2} \\
	&\leq (2C+1)\varepsilon
\end{align*}
for a constant $C$ depending only on $k$ by Theorem \ref{thm:shadrin}.
This means that
$P_n g$ converges in $L^1$ to some limit that we will again call $P_\infty g$. It actually coincides with the operator $P_\infty$ on $L^2$ and satisfies the same $L^1$ bound as the sequence $(P_n)$.
Summing up we have 
\begin{equation*}
	\| P_n g - P_\infty g\|_{L^1} \to 0,
\end{equation*}
for any $ g \in L^1. $ Applying Lemma \ref{lem:extension} to $(P_n -P_\infty)$ 
we obtain the following  vector valued extension.  For any Banach space $X$ 
\begin{equation*}
	\| P_n g - P_\infty g\|_{L^1_X} \to 0,
\end{equation*}
for  $ g \in L^1_X . $

The next step is to show pointwise convergence of $P_ng$ for continuous
functions~$g$.
We define $U$ to be the complement of the  set of all accumulation points 
of the given knot sequence $(t_i)$.
This set $U$ is open, so it can be
written as a disjoint union of open intervals
\begin{equation*}
	U=\cup_{j=1}^\infty U_j.
\end{equation*}

\begin{lem}\label{lem:aecontinuous}
	Let $g\in C[0,1]$. Then, $P_n g$ converges pointwise a.e. to $P_\infty g$ with respect
	to Lebesgue measure.
\end{lem}
\begin{proof}
	We first show that on each interval $U_j$, $P_n g$ converges locally
	uniformly. Let $A\subset U_j$ be a compact subset. 
Then the definition of $U_j$ implies that  $s:= \inf\{ \lambda(I_n(t)) : t\in A, n\in\mathbb N\}$ is positive. 
	Observe that of course, since in particular $g\in L^1[0,1]$, the
	sequence $P_n g$ converges in $L^1$.
Therefore, for $\varepsilon>0$, we can choose $M$ so large that for all $n,m\geq M$, $\|P_n g -
P_m g\|_{L^1} \leq \varepsilon s$.
We then estimate by Theorem \ref{thm:boundgram} for $n\geq m\geq M$ and $t\in
A$:
	\begin{align*}
		|(P_n - P_m)g(t)| &= |P_n (P_n - P_m)g(t)| \\
		&= \Big|\sum_{i}  \langle (P_n - P_m) g,
		N_i^{(n)}\rangle N_i^{(n)*}(t)\Big| \\
		&\lesssim \sum_{i} \frac{q^{|i-i_n(t)|}}{\lambda(I_n(t))}
		\|(P_n - P_m) g\|_{L^1(\supp N_i^{(n)})} \\
		&\leq \|(P_n - P_m)g\|_{L^1([0,1])} \sum_{i}
		\frac{q^{|i-i_n(t)|}}{s} \\
		&\lesssim
		\frac{\|(P_n-P_m)g\|_{L^1[0,1]}}{s} \leq \varepsilon,
	\end{align*}
	so $P_n g$ converges uniformly on $A$.

	If $t\in U^c$, we can assume that on both sides of $t$, there is a subsequence of grid points
	converging to $t$, since if there is a side that does not have a
	sequence of grid points converging to $t$, the point $t$ would be an endpoint
	of an interval $U_j$ and the union over all endpoints of $U_j$ is 
	countable and therefore a Lebesgue zero set.
	Let $\varepsilon>0$ and let $\ell$ be such that 
\begin{equation}\label{3-11-17-3}q^\ell \|g\|_{L^\infty} \leq
	\varepsilon .\end{equation}
 We choose $M$ so large that for any $m\geq M$ on each
	side of $t$ there are $\ell$ grid points of $\Delta_m$ and each of
	those grid point intervals has the property that the length is $<\delta$
	with $\delta>0$ being such that $\omega_k(g,\delta)<\varepsilon$,
	where $\omega_k$ is the $k$th modulus of smoothness. With
	this choice, by Theorem \ref{thm:jackson}, 
	there exists  a function $f\in S_{M}$ with $\|f\|_{L^\infty}\lesssim
	\|g\|_{L^\infty}$ that approximates $g$ well on
	the smallest interval  $B$ that contains  $\ell-k$ grid points to the left of $t$  and  
        $\ell-k$ grid points to the
	right of $t$ in $\Delta_M$ in the sense that 
	\begin{equation}\label{3-11-17-2}
		\| f - g\|_{L^\infty(B)} \lesssim \omega_k(g,\delta) \leq
		\varepsilon.
	\end{equation}
	Therefore, we can write for $n,m\geq M$
	\begin{align*}
		(P_n - P_m)g(t) &= P_n(g-f)(t) + P_m(f-g)(t) .
	\end{align*}
	Next, estimate $P_n(g-f)(t)$ for $n\geq M$ by Theorem \ref{thm:boundgram}:
	\begin{align*}
		|P_n(g-f)(t)| &= \Big|\sum_{i} \langle g-f,
		N_i^{(n)}\rangle N_i^{(n)*}(t)\Big| \\
		&\lesssim  \sum_{i}
		\|g-f\|_{L^\infty(\supp(N_i^{(n)}))} \lambda(\supp
		N_i^{(n)})\frac{q^{|i-i_n(t)|}}{\lambda(\supp N_i^{(n)})} \\
		&= 
		\sum_i
		q^{|i-i_n(t)|}\|g-f\|_{L^\infty(\supp N_i^{(n)})}.
	\end{align*}
In estimating the above series we  distinguish two  cases  for the value of $ i$:	
$$|i-i_n(t)|\leq \ell-2k, \qquad\text{and}\qquad |i-i_n(t)|> \ell-2k. $$  
Using  $ \| g-f \|_{L^\infty(\supp N_i^{(n)})} \leq \|g-f
	\|_{L^\infty(B)}$ and     \eqref{3-11-17-2} we get 
$$	\sum_{i: |i-i_n(t)|\leq \ell-2k}
		q^{|i-i_n(t)|}\|g-f\|_{L^\infty(\supp N_i^{(n)})} \lesssim \varepsilon .$$
Using $\|g-f\|_{L^\infty(\supp N_i^{(n)})} \lesssim  \|g\|_{L^\infty} $ and  \eqref{3-11-17-3} gives 
$$	\sum_{i: |i-i_n(t)| >  \ell-2k}
		q^{|i-i_n(t)|}\|g-f\|_{L^\infty(\supp N_i^{(n)})}  \lesssim \varepsilon .$$
This yields $|P_n(g-f)(t)|\lesssim \varepsilon$ for $n\geq M$ and
	therefore
	$P_n g(t)$ converges as $n\to\infty$.
\end{proof}

The following theorem establishes   the spline analogue of the 
martingale results  (M1) and (M3).
The role of Lemma \ref{lem:aecontinuous} in the proof given below is to free 
the main theorem  in \cite{PassenbrunnerShadrin2014} from the restriction  that  
 the sequence of knots $(t_i)$ is dense in $[0,1]$. 

\begin{thm}\label{thm:aeconv}
Let $X$ be any Banach space. 	For $f\in L^1_X$, 
there exists $ E \subset [0,1] $ with $ \lambda (E) = 0 $ such that 
$$ \lim_{n\to \infty}  P_n f ( t) = P_{\infty} f(t) , $$ 
for any $ t \in  [0,1] \setminus E . $
\end{thm}
\begin{proof}
The proof uses standard arguments
	involving Lemma \ref{lem:aecontinuous}, Theorems \ref{thm:maximal} and
	\ref{thm:weaktype}.
	(See 	\cite{PassenbrunnerShadrin2014}.)

       \textsc{Step 1:} (The scalar case.)  Fix  $ v \in L^1$  and $ \ell \in \bN . $ Put 
  \begin{equation*}          
  A^{(\ell)} (v) = \bigcap_N \bigcup_{m, n \ge N } \{t: | P_n v(t) - P_m v (t) | > 1/\ell \} . 
  \end{equation*} 
        By Lemma \ref{lem:aecontinuous},  for any $ u \in C[0,1] , $ 
 \begin{equation*}   
\lambda (A^{(\ell)} (v) )=   \lambda (A^{(\ell)} (v-u)) 
 \end{equation*}
Let $P^*(v -u) (t) = \sup_{n} |P_n (v -u)(t)|$. Clearly we have 
\begin{equation*}          
\lambda (A^{(\ell)} (v-u))  \le     \lambda  (\{t:2 P^*(v-u)(t) \ge 1/\ell \} ).
\end{equation*}
 By Theorem~\ref{thm:maximal},  $P^*$ is dominated pointwise by the Hardy-Littlewood
	maximal function and the latter is of weak type 1-1. Hence 
\begin{equation*}
       \lambda( \{t:P^*(v-u)(t) \ge 1/\ell \} ) \lesssim  \ell \|v -u \|_{L^1} 
	\end{equation*}
	
 Now fix   $ \varepsilon  > 0 $.   Since   $C[0,1]$ is dense in  $L^1$,
        there exists    $u \in C[0,1]$ such that $\|v -u \|_{L^1} \le \varepsilon /\ell . $
Thus,  we obtained 
$\lambda (A^{(\ell)} (v) ) < \varepsilon$ for any $ \varepsilon > 0 $, or  $\lambda (A^{(\ell)} (v) ) =0 .$
It remains to observe that   
        \begin{equation*}
\lambda (\{t:P_nv(t) \text{ does not converge}\}) = \lambda \Big(\bigcup_\ell  A^{(\ell)} (v) \Big) = 0. 
        \end{equation*}
        
\textsc{Step 2:}(Vector valued extension.)
Let $g_m = v_m \otimes x_m $ where $v_m \in L^1 $ and $ x_m \in X $
and let $ g  \in L^1\otimes X$ be given as 
 \begin{equation*}
g = \sum_{m= 1 }^M g_m . 
 \end{equation*} 
Applying  Step 1 to  $v_m$ 
shows that $P_n g(t)$ converges in $X$ for $\lambda $-almost every $t\in [0,1] $.
Taking into account that  $L^1\otimes X$ is dense in $L^1_X$, 
we may now repeat the argument above  to finish the proof. Details are as follows:  
	Fix  $ f \in L_X^1$  and $ \ell \in \bN . $ Put 
  \begin{equation*}          
  A^{(\ell)} (f) = \bigcap_N \bigcup_{m, n \ge N } \{t: \| P_n f(t) - P_m f (t) \|_X > 1/\ell \} . 
  \end{equation*}
Then 
 \begin{equation*}
\lambda	( \{t:P_nf(t) \text{ does not converge in $X$}\}) = 
\lambda \Big(\bigcup_\ell  A^{(\ell)} (f) \Big) . 
        \end{equation*}
It remains to prove   that  $ \lambda ( A^{(\ell)} (f) ) =0 .$	
To this end observe that  for   $ g  \in L^1\otimes X$ we have 
$
\lambda (A^{(\ell)} (f) )=   \lambda (A^{(\ell)} (f-g)). 
$
Define the maximal function  $P^*(f -g) (t) = \sup_n \|P_n (f -g)(t)\|_X$. Clearly we have 
\begin{equation*}          
\lambda (A^{(\ell)} (f-g))  \le     \lambda (\{t:2 P^*(f-g)(t) \ge 1/\ell \}).
\end{equation*}
 By Theorem~\ref{thm:maximal}, and the weak type 1-1 estimate for  the Hardy-Littlewood
	maximal function,
\begin{equation*}
        \lambda (\{t:P^*(f-g)(t) \ge 1/\ell \}) \lesssim  \ell \|f -g \|_{L_X^1}. 
	\end{equation*}
Fix   $ \varepsilon  > 0 $, choose 
       $g \in L^1\otimes X$ such that $\|f -g \|_{L^1_X} \le \varepsilon /\ell . $
This gives  $\lambda( A^{(\ell)} (f) ) \lesssim\varepsilon$ for any $ \varepsilon  > 0 $, 
proving that    $ \lambda( A^{(\ell)} (f) ) =0 .$	
\end{proof}

\section{B-spline constructions}
\label{sec:splineConstruction}
Recall that we defined $U$ to be the complement of the  set of all accumulation points 
of the sequence $(t_i)$.
This set $U$ is open, so it can be
written as a disjoint union of open intervals
\begin{equation*}
	U=\cup_{j=1}^\infty U_j.
\end{equation*}
Observe that, since a boundary point $a$ of $U_j$ is an accumulation point of the
sequence $(t_j)$, there exists a subsequence of grid points converging to $a$.
Let 
\begin{align*}
	B_j := \{ a\in\partial U_j : \text{ there is {\em no} sequence of grid points
	} \\
	\text{contained in $U_j$ that converges to $a$} \}
\end{align*}
Now we set $V_j := U_j \cup B_j$ and $V := \cup_{j} V_j$.

Consider an arbitrary interval $V_{j_0}$ and set $a=\inf V_{j_0}$, $b=\sup
V_{j_0}$. We define the sequences $(s_j)$ and $(s_j^{(n)})$ -- rewritten in increasing order with  multiplicities included --
to be the   points
in  $(t_j)$ and $(t_j)_{j=1}^n$, respectively, 
that are contained in  $V_{j_0}$.
If $a\in V_{j_0}$, the sequence $(s_j)$ is finite to the left
and we extend the sequences $(s_j)$ and $(s_j^{(n)})$ so that they contain 
the point $a$ $k$ times
and they are still increasing. 
Similarly, if $b\in V_{j_0}$, the sequence $(s_j)$ is finite to the right
and we extend the sequences $(s_j)$ and $(s_j^{(n)})$ so that they contain 
the point $b$ $k$ times
and they are still increasing. 
Observe that if $a\notin V_{j_0}$ or $b\notin V_{j_0}$, the sequence $(s_j)$ is 
infinite to the left or infinite to the right, respectively.
We choose the indices of the sequences $(s_j)$ and $(s_j^{(n)})$ so
that for fixed $j$ and $n$ sufficiently large, we have $s_j = s_j^{(n)}$.
Let $(\bar N_j)$ and $(\bar N_j^{(n)})$ be the sequences of B-spline
functions corresponding to the sequences $(s_j)$ and $(s_j^{(n)})$,
respectively.
Observe that the choice of the sequences $(s_j)$ and $(s_j^{(n)})$ implies for
all $j$ that 
$\bar N_j \equiv \bar N_j^{(n)}$ if $n$ is sufficiently large.
Let $(N_j^{(n)})$ be the sequence of those B-spline functions from Section
\ref{sec:prelims} whose supports intersect the set $V_{j_0}$ on a set of
positive Lebesgue measure, but do not contain 
any of the points $\partial U_{j_0}\setminus B_{j_0}$ and without loss of
generality, we assume that this sequence is enumerated in such a way that
starting index and ending index coincide with the ones of the sequence $(\bar
N_j^{(n)})_j$. Then, the relation between $(N_j^{(n)})_j$ and $(\bar N_j)_j$ is
given by the following lemma:

\begin{lem}\label{lem:convBspline}
	For all $j$, the sequence of functions $(N_{j}^{(n)}\charfun_{V_{j_0}})$ converges
	uniformly to some function that coincides with $\bar N_{j}$ on $U_{j_0}$.
\end{lem}
\begin{proof}
	If the support of $N_i^{(n)}$ is a subset of $V_{j_0}$ for sufficiently
	large $n$, the sequence 
	$n\mapsto N_i^{(n)}$ is eventually constant and coincides by definition
	with $\bar N_i$. In the other case, this follows by the recursion
	formula (Theorem \ref{thm:recursion}) 
	for B-splines and observing that  for piecewise linear B-splines, this is 
	clear.  
\end{proof}
In view of the above lemma, we may assume that $\bar N_i$ coincides with 
the uniform limit of the sequence $(N_i^{(n)}\charfun_{V_{j_0}})$.
Define $(\bar N_j^{(n)*})$ to be the dual B-splines to
$(\bar N_j^{(n)})$.
For $t\in [0,1]$ denote by $\bar I_n(t)$ a smallest grid point interval of positive
length in the grid $( s_j^{(n)})$
that contains the point $t$.  We denote by  $\bar i_n(t)$ the largest index $i$ such that
$\bar I_n(t)\subset \supp \bar N_i^{(n)}$. 
Additionally, denote by $\bar h_{ij}^{(n)}$ the length of the convex
hull of the union of the supports of $\bar N_{i}^{(n)}$ and $\bar N_{j}^{(n)}$. 
Similarly  we let  $\bar I(t)$  denote a smallest grid point interval of positive
length in the grid $( s_j)$ containing  $t \in [0,1]$.  
We denote by  $\bar i(t)$ the largest index $i$ such that
$\bar I(t)\subset \supp \bar N_i$. 
Next, we identify dual functions to the sequence $(\bar N_j)$:
\begin{lem}\label{lem:splineDual}
	For each $j$, the sequence $\bar N_j^{(n)*}$ converges 
	uniformly on each interval $[s_i,s_{i+1}]$
	to some function $\bar N_j^*$ that satisfies 
	\begin{enumerate}
		\item 
		$\langle \bar N_j^*,\bar N_i\rangle = \delta_{ij}$ for all $i$,
	\item for all $t\in U_{j_0}$, 
		\begin{equation}\label{eq:barNDual}
			|\bar N_j^*(t)| \lesssim \frac{q^{|j-\bar
			i(t)|}}{\lambda(\bar
			I(t))}, 
		\end{equation}
	\end{enumerate}
where $q \in (0,1) $ is given by Theorem \ref{thm:boundgram}.
\end{lem}
\begin{proof}
	We fix the index $j$, the point $t\in U_{j_0}$ and $\varepsilon>0$. 
	Next, we choose $M$ sufficiently large so that for all $m\geq M$ and all
	$\ell$ with the property $|\ell - \bar i(t)| \leq L$ we have
	$s_\ell^{(m)}= s_\ell$, where $L$ is chosen so that $q^L/\lambda(\bar
	I(t))\leq \varepsilon$ and $|j-\bar i(t)|\leq L-k$.
	For $n\geq m\geq M$, we can expand the function $\bar N_j^{(m)*}$ in the
	basis $(\bar N_i^{(n)*})$ and write
	\begin{equation}\label{eq:expansionDual}
		\bar N_j^{(m)*} = \sum_{i} \alpha_{ji} \bar N_i^{(n)*}.
	\end{equation}
We now turn to estimating the coefficients  $\alpha_{ji}$ defined by  
equation \eqref{eq:expansionDual}.
Observe that for $\ell$ with $|\ell - \bar i(t)|\leq L-k$, we have $\bar
	N_\ell^{(m)}\equiv \bar N_\ell^{(n)}$, and therefore, for such $\ell$,
	\begin{align*}
		\delta_{j\ell} &= \langle \bar N_j^{(m)*}, \bar N_\ell^{(m)}\rangle =
		\langle \bar N_j^{(m)*}, \bar N_\ell^{(n)}\rangle 
		=\sum_{i} \alpha_{ji} \langle\bar N_i^{(n)*}, \bar
		N_\ell^{(n)} \rangle = \alpha_{j\ell},
	\end{align*}
	which means that the expansion \eqref{eq:expansionDual} takes the
	form
	\begin{equation}
		\label{eq:dualestimate}
		\bar N_j^{(m)*} = \bar N_j^{(n)*} + \sum_{\ell:|\ell-\bar i(t)| > L-k}
		\alpha_{j\ell} \bar N_\ell^{(n)*}.
	\end{equation}

	Next we show that $|\alpha_{j\ell}|$ is bounded by a constant independently
	of $j,\ell$ and $m,n$. Recall   $\bar h_{ij}^{(m)} $ denotes  the length 
of the smallest interval containing     $\supp \bar N_i^{(m)} \cup \supp \bar N_j^{(m)}. $ 
By Theorem \ref{thm:boundgram}, applied to the matrix $(\bar
	a_{ij}^{(m)})=(\langle \bar N_{i}^{(m)*}, \bar N_{j}^{(m)*}\rangle)$,  we get
\begin{align*}
		|\alpha_{j\ell}| &= |\langle \bar N_j^{(m)*}, \bar
		N_\ell^{(n)}\rangle| =
		\Big|\Big\langle \sum_{i} \bar a_{ij}^{(m)} \bar N_i^{(m)}, \bar
		N_\ell^{(n)} \Big\rangle\Big| \\
		&\lesssim \sum_i \frac{q^{|i-j|}}{\bar h_{ij}^{(m)}} \langle
		\bar N_i^{(m)}, \bar N_\ell^{(n)}\rangle \leq 
		\sum_i \frac{q^{|i-j|}}{\bar h_{ij}^{(m)}} 
		\lambda(\supp \bar N_i^{(m)}) \\
		&\leq \sum_i q^{|i-j|} \lesssim 1.
	\end{align*}
	This can be used to obtain an estimate for the difference between $\bar 
	N_j^{(m)*}(t)$ and $\bar N_j^{(n)*}(t)$ by inserting it into
	\eqref{eq:dualestimate} and applying again Theorem \ref{thm:boundgram}: 
	\begin{align*}
		| (\bar N_j^{(m)*} - \bar N_j^{(n)*} )(t) | &\leq \sum_{\ell:
			|\ell-
		\bar i(t)| > L - k} |\alpha_{j\ell}| |\bar N_\ell^{(n)*}(t)| \\
		&\lesssim \sum_{\ell: |\ell-\bar i(t)|> L -k} 
			\frac{q^{|\ell-\bar i_n(t)|}}{\lambda(\bar I_n(t))} 
			\lesssim \frac{q^L}{\lambda(\bar I(t))} \leq \varepsilon.
	\end{align*}
	This finishes the proof of the convergence part. 
	Estimate \eqref{eq:barNDual} now follows from the corresponding estimate
	for $\bar N_j^{(n)*}$ in Theorem \ref{thm:boundgram}.

	Now, we turn to the proof of property $(1)$. 
	Let $j,i$ be arbitrary. Choose $M$ sufficiently large so that for all
	$n\geq M$, we have $\bar N_i \equiv \bar N_i^{(n)}$ on $U_{j_0}$, therefore, 
	\begin{align*}
| \langle \bar N_j^*, \bar
			N_i\rangle
			-\delta_{ij}| &=	
		| \langle \bar N_j^*, \bar
			N_i\rangle - \langle \bar N_j^{(n)*} , \bar
			N_i^{(n)}\rangle | 
			=|\langle \bar N_j^* - \bar N_j^{(n)*} , \bar
			N_i^{(n)}\rangle| \\
			&\leq \| \bar N_j^{*} - \bar N_j^{(n)*}
			\|_{L^\infty(\supp \bar N_i^{(n)})}\cdot \lambda(\supp \bar
			N_i^{(n)}),
	\end{align*}
	which, by the local uniform convergence of $\bar N_j^{(n)*}$ to $\bar
	N_j^*$, tends to zero.
\end{proof}

\section{A measure estimate}\label{sec:measure}
Let $ \sigma$ be a measure defined on the unit interval. Recall that $P_n(\sigma) $ 
is defined by duality. In view of Theorem~\ref{thm:boundgram}, localized and  pointwise estimates for   $P_n(\sigma) $ are controlled by  terms of the form 
 $$ \sum_{i,j} \frac{q^{|i-j|}}{h_{ij}^{(n)}}
		|\sigma|(\supp N_i^{(n)}) N_j^{(n)}.$$
Subsequently the following Lemma will be used to show that  $P_n(\sigma) $  converges a.e. to zero, 
for any measure $ \sigma$ singular to the Lebesgue measure.
\begin{lem}\label{31-10-17-3}
	Let $F_r$ be a Borel subset of $V^c$ and $\theta$ a positive measure on
	$[0,1]$ with
	$\theta(F_r)=0$
	so that for all $x\in F_r$, we have
	\begin{equation*}
		\limsup_n b_n(x)>1/r,
	\end{equation*}
	where $b_n(x)$ is a positive function satisfying
	\begin{equation*}
		b_n(x) \lesssim \sum_{i,j} \frac{q^{|i-j|}}{h_{ij}^{(n)}}
		\theta(\supp N_i^{(n)}) N_j^{(n)}(x),\qquad x\in F_r.
	\end{equation*}

	Then, $\lambda(F_r)=0$.
\end{lem}
\begin{proof}
	First observe that we can assume that each point in $F_r$ can be
	approximated from both sides with points of the sequence $(t_i)$, since
	the set of points in $V^c$ for which this is not possible is a subset of
	$\cup_j \partial V_j$ and therefore of Lebesgue measure zero. 

\textsc{Step 1:}
For an arbitrary positive number $\varepsilon$, by the regularity of $\theta$, we
can take an open set $U_\varepsilon\subset [0,1]$ with  $U_\varepsilon \supset F_r$  
and $\theta(U_\varepsilon)\leq \varepsilon$. 
Then, for $x\in F_r$, we choose a ball $B_x \subset
U_\varepsilon$ with center $x$, define $s_m(x) = \{j : N_j^{(m)}(x)\neq 0\}$
and calculate
\begin{align*}
	b_m(x) & \lesssim  \sum_{i,j} \frac{q^{|i-j|}}{h_{ij}^{(m)}}\theta(\supp
	N_i^{(m)}) N_j^{(m)}(x)\\
	&\lesssim  \sum_{j\in s_m(x)}\sum_i
	\frac{q^{|i-j|}}{h_{ij}^{(m)}}\theta(\supp N_i^{(m)})
	\\
	&\lesssim  \max_{ j \in s_m(x) } \sum_{i}
	\frac{q^{|i-j|}}{h_{ij}^{(m)}}
	\theta(\supp N_i^{(m)}) \\
	&= C \max_{j\in s_m(x)} (\Sigma_{1,j}^{(m)} +
	\Sigma_{2,j}^{(m)}),
\end{align*}
for some constant $C$ and 
where 
\begin{equation*}
\Sigma_{1,j}^{(m)} := \sum_{i\in \Lambda_1^{(m)}} \frac{q^{|i-j|}}{h_{ij}^{(m)}}
\theta(\supp N_i^{(m)}),\qquad \Sigma_{2,j}^{(m)} := \sum_{i\in \Lambda_2^{(m)}} \frac{q^{|i-j|}}{h_{ij}^{(m)}}
\theta(\supp N_i^{(m)})
\end{equation*}
and 
\begin{equation*}
\Lambda_1^{(m)} = \{i: \supp N_i^{(m)}\subset B_x\},\qquad \Lambda_{2}^{(m)} = (\Lambda_1^{(m)})^c
\end{equation*}

\textsc{Step 2:} Next, we show that it is possible to choose $m$ sufficiently large 
to have $\Sigma_{2,j}^{(m)} \leq 1/(2Cr)$ for all $j\in s_m(x)$.

To do that, let $j_m\in s_m(x)$ and
observe that
\begin{align*}
\Sigma_{2,j_m}^{(m)} &= \sum_{i\in \Lambda_2^{(m)}}
\frac{q^{|i-j_m|}\theta(\supp N_i^{(m)})}{h_{ij_m}^{(m)}} 
\leq \sum_{i\in \Lambda_2^{(m)}}
\frac{q^{|i-j_m|}\theta(\supp N_i^{(m)})}{d(x,\supp N_i^{(m)})} =: A_{2,j_m}^{(m)},
\end{align*}
where $d(x,\supp N_i^{(m)})$  denotes the Euclidean distance between $x$ and $\supp N_i^{(m)}$.
Now, for $n>m$ sufficiently large, we get
\begin{align}\nonumber
A_{2,j_n}^{(n)} &= \sum_{\ell\in \Lambda_2^{(n)}}
\frac{q^{|\ell-j_n|}\theta(\supp N_{\ell}^{(n)})}{d(x,\supp N_{\ell}^{(n)})} \\
\label{31-10-17-2}&\leq \sum_{i\in\Lambda_2^{(m)}} \sum_{\substack{\ell\in\Lambda_2^{(n)}, \\ \supp N_\ell^{(n)}\subset
		\supp N_i^{(m)}}}
\frac{q^{|\ell-j_n|}\theta(\supp N_{\ell}^{(n)})}{d(x,\supp N_{\ell}^{(n)})} .
\end{align}
Define  $L_{n,m} $ to be the cardinality of  the set $ \{t_i : m < {i} \le n\}  \cap B_x \cap [0,x]
  $ 
and  $R_{n,m} $  the cardinality of  $ \{t_i : m < {i} \le n\}  \cap B_x \cap [x,1] . $ Put 
$$  K_{n,m} = \min \{L_{n,m}, R_{n,m}\} . $$ 
The term \eqref{31-10-17-2} admits the following upper bound 
\begin{align*}
&q^{K_{n,m}}\sum_{i\in\Lambda_2^{(m)}}\frac{q^{|i-j_m|}}{d(x,\supp N_i^{(m)})} \sum_{\substack{\ell\in\Lambda_2^{(n)}, \\ \supp N_\ell^{(n)}\subset
		\supp N_i^{(m)}}} \theta(\supp N_\ell^{(n)}) \\
&\lesssim q^{K_{n,m}}\sum_{i\in\Lambda_2^{(m)}}\frac{q^{|i-j_m|}}{d(x,\supp N_i^{(m)})} \theta(\supp N_i^{(m)})
= q^{K_{n,m}} A_{2,j_m}^{(m)},
\end{align*}
Since $x$ can be approximated by grid points from both sides, $\lim_{n\to\infty} K_{n,m} = \infty$, and we can choose $m$ sufficiently large
to guarantee
\begin{equation*}
\Sigma_{2,j}^{(m)} \leq A_{2,j}^{(m)} \leq \frac{1}{2Cr}.
\end{equation*}

\textsc{Step 3:} Next, we show that for any $x\in F_r$, 
there exists an open interval $C_x\subset B_x$ such that $\theta(C_x)/\lambda(C_x)\gtrsim 1/(2Cr)$.

By Step~2 and the fact that $\limsup b_n(x)>1/r$ for $x\in F_r$,
there exists an integer $m$ and an index $j_0\in s_m(x)$ with
\begin{equation*}
	\Sigma_{1,j_0}^{(m)} \geq \frac{1}{2Cr},
\end{equation*}
which means that 
\begin{align*}
	\frac{1}{2Cr}&\leq \sum_{i \in \Lambda_1^{(m)}}
	\frac{q^{|i-j_0|}}{h_{ij_0}^{(m)}}\theta(\supp N_i^{(m)}) \\
	&\leq  \sum_{i\in\Lambda_1^{(m)}}
	\frac{q^{|i-j_0|}}{h_{ij_0}^{(m)}}\theta\big(\co(\supp N_i^{(m)}\cup
	\supp N_{j_0}^{(m)})\big),
\end{align*}
where $\co(A)$ denotes the convex hull of the set $A$.
Since $\sum_{i\in\Lambda_1^{(m)}}
	q^{|i-j_0|} \lesssim 1, $ 
there exists a constant $c$ depending only on $q$ and an index $i$
with $\supp N_i^{(m)} \subset B_x$ and
\begin{equation*}
	\frac{\theta\big( \co(\supp N_i^{(m)}\cup \supp N_{j_0}^{(m)})\big)}{h_{ij_0}^{(m)}} \geq
	\frac{c}{2Cr},
\end{equation*}
which means that there exists an open interval $C_x$ with $x\in C_x\subset B_x$
with the property $\theta(C_x) / \lambda(C_x) \geq c/(2Cr)$.

\textsc{Step 4:}
Now we finish with a standard argument using the Vitali covering lemma (Lemma
\ref{lem:vitali}): there
exists a countable collection $J$ of points $x\in F_r$ such that $\{ C_x : x\in
J\}$ are disjoint sets and
\begin{equation*}
	F_r \subset \bigcup_{x\in F_r} C_x \subset \bigcup_{x\in J} 5C_x.
\end{equation*}
Combining this with Steps 1-3, we conclude
\begin{align*}
	\lambda(F_r) \leq \lambda\Big(\bigcup_{x\in J}
	5C_x\Big) \leq 5\sum_{x\in J} \lambda(C_x) \leq \frac{10Cr}{c}\sum_{x\in J}
	\theta(C_x) \leq \frac{10Cr}{c}\theta(U_\varepsilon) \leq
	\frac{10Cr}{c}\varepsilon.
\end{align*}
Since this inequality holds for all $\varepsilon>0$, we get that $\lambda(F_r)=0$.
\end{proof}

\section{Proof of the Spline Convergence Theorem}
\label{sec:proof}

In this section, we prove the Spline Convergence Theorem \ref{thm:main}.
For $f\in S_m$, a consequence of  \eqref{eq:projidentity} is 
\begin{align*}
\int_0^1 g_n(t) \cdot f(t) \dif\lambda(t) &= \int_0^1 g_n(t)\cdot P_m
f(t)\dif\lambda(t) = \int P_m g_n(t) \cdot f(t)\dif \lambda(t)\\
&= \int_0^1 g_m(t)
\cdot f(t)\dif\lambda(t),\qquad n\geq m.
\end{align*}
This means in particular that for all $f\in \cup_n
S_n$, the limit of $\int_0^1 g_n(t)\cdot f(t)\dif\lambda(t)$
exists, so we can define the linear operator 
\begin{equation*}
	T : \cup S_n \to X,\qquad  f\mapsto \lim_n \int_0^1 g_n(t)\cdot f(t) \dif\lambda(t).
\end{equation*}
By Alaoglu's theorem, we may choose a  subsequence $k_n$ such that the bounded sequence of measures $\|g_{k_n}\|_X
\dif\lambda$ converges in the weak*-topology to some scalar measure $\mu$.
Then, as each $f\in\cup_n S_n$ is continuous,
\begin{equation}
	\label{eq:uniformestimate}
	\| Tf\|_X \leq \int_{0}^1 |f(t)|\dif \mu(t),\qquad f\in \cup S_n.
\end{equation}
We let $W$ denote the $L^1([0,1] , \mu)$-closure of $\cup_n S_n$. By \eqref{eq:uniformestimate}, the operator $T$  extends to $W$    with  norm  bounded by $1. $ 

We set
\begin{equation*}
	(P_n T)(t) := \sum_{i} (T N_i^{(n)}) N_i^{(n)*}(t)
\end{equation*}
which is well defined. Moreover,
\begin{align*}
	(P_n T)(t) &= \sum_{i} (T N_i^{(n)}) N_{i}^{(n)*}(t) \\
	&=\sum_{i} \lim_m \int g_m N_i^{(n)} \dif\lambda \cdot N_i^{(n)*}(t) \\
	&=\sum_{i} \langle g_n,N_i^{(n)}\rangle N_i^{(n)*}(t) = (P_n g_n)(t) = g_n(t).
\end{align*}
Thus  we verify  a.e. convergence of $g_n$,  by   
showing   a.e. convergence
of $P_nT$ below.

\begin{lem}\label{lem:charfun}
	For all $f\in \cup S_n$, the function $f\charfun_{V_j}$ is contained in
	$W$ and also $f\charfun_V$ is contained in $W$.
	Additionally, on  the complement of $V=\cup V_j$, the $\sigma$-algebra $\mathcal F =
	\{A\in\mathcal B : \charfun_A \in W\}$ coincides with the Borel $\sigma$-algebra $\mathcal B$,
	i.e., $V^c \cap \mathcal F = V^c\cap \mathcal B$.
\end{lem}
\begin{proof}
Since $W$ is a linear space, it suffices to show the
assertion for each B-spline function $N_i^{(m)}$ contained in some $S_m$. 
By Corollary \ref{cor:boehm}, it can be written as a linear combination of finer B-spline
functions ($n\geq m$)
\begin{equation*}
	N_i^{(m)} = \sum_\ell \lambda_\ell^{(n)} N_\ell^{(n)},
\end{equation*}
where each coefficient $\lambda_\ell^{(n)}$ satisfies the inequality 
$|\lambda_\ell^{(n)}|\leq 1$. 
 We set $$h_n :=
\sum_{\ell\in\Lambda_n} 
	 \lambda_\ell^{(n)} N_\ell^{(n)},$$
	 where the index set $\Lambda_n$ is defined to contain precisely those
	 indices $\ell$ so that $\supp N_\ell^{(n)}$ intersects $V_j$ but does
	 not contain any of the points $\partial U_j \setminus B_j$. The
	 function $h_n$ is contained in $S_n$ and satisfies $|h_n|\le 1$. Observe
	 that $\supp h_n \subset O_n$ for some open set $O_n$ and $h_n \equiv
	 N_i^{(m)}$ on some compact set $A_n\subset V_{j}$ that satisfy
	 $O_n\setminus A_n \downarrow \emptyset$ as $n\to \infty$ and thus, 
	\begin{equation*}
		\|  N_i^{(m)}\charfun_{V_{j}} - h_n \|_{L^1(\mu)} \lesssim \mu
		(O_n\setminus A_n)\to 0.
	\end{equation*}
	This shows that $N_i^{(m)}\charfun_{V_j}\in W$.
	
	Since $\mu$ is a finite measure, $\lim_{n} \mu(\cup_{j\geq n} V_j)=0$,
	and therefore, $f\charfun_V = f\charfun_{\cup_j V_j}$ is also contained
	in $W$.

	Similarly, we see that the collection $\mathcal F =\{A\in\mathcal B : \charfun_A \in W\}$
 is a $\sigma$-algebra. So, in
	order to show $V^c\cap \mathcal F = V^c\cap \mathcal B$ we will show
	that for each interval $(c,d)$ contained in $[0,1]$, we can find an interval $I\in
	\mathcal F$ with the property $V^c\cap (c,d) = V^c\cap I$.
	By the same reasoning as in the approximation of
	$N_i^{(m)}\charfun_{V_j}$
	by finer spline functions, we can give the following sufficient
	condition for an interval $I$ to be contained in $\mathcal F$:
if for all $a\in \{\inf I, \sup I\}$ we have either
\begin{equation*}
a\in I \text{ and there exists a seq. of grid points conv. from outside of $I$ to $a$}
\end{equation*}
or
\begin{equation*}
a\notin I \text{ and there exists a seq. of grid points conv. from inside of $I$ to $a$},
\end{equation*}
then $I\in\mathcal F$. 
Let now $(c,d)$ be an arbitrary interval and assume first that $c,d\notin \cup_j
\partial U_j$. For arbitrary points $x\in [0,1]$, define
\begin{equation*}
	I(x) := \begin{cases}
		V_j, &\text{if }x\in U_j, \\
		\emptyset, &\text{otherwise}.
	\end{cases}
\end{equation*}
Then, by the above sufficient criterion, the set $I = (c,d)\setminus(I(c)\cup
I(d))$ is contained in $\mathcal F$.
Moreoever, $V^c\cap (c,d) = V^c\cap I$ and this shows
that $(c,d)\cap V^c \in \mathcal F\cap V^c$.
In general, since the set $\cup_j \partial U_j$ is countable, we can find
sequences $c_n\geq c$ and $d_n\leq d$ with $c_n,d_n \notin \bigcup_j \partial
U_j$, $c_n\to c$, $d_n\to d$, and 
\begin{equation*}
	(c,d)\cap V^c = (\cup_n (c_n, d_n))\cap V^c \in\mathcal F\cap V^c,
\end{equation*}
since $\mathcal F\cap V^c$ is a $\sigma$-algebra. This shows the fact that
$\mathcal F\cap V^c = \mathcal B\cap V^c$.
\end{proof}

\begin{proof}[Proof of Theorem \ref{thm:main}]
\textsc{Part 1: $t\in V^c$:}
By Lemma \ref{lem:charfun}, we can decompose
\begin{align*}
	g_n(t) &= (P_n T)(t) = \sum_{i} T (N_i^{(n)}) N_i^{(n)*}(t) \\
	&= \sum_{i}  T(N_i^{(n)} \charfun_V) N_i^{(n)*}(t) +
	\sum_{i}  T(N_i^{(n)} \charfun_{V^c}) N_i^{(n)*}(t) \\
	&=: \Sigma_1^{(n)}(t) + \Sigma_2^{(n)}(t).
\end{align*}

\textsc{Part 1.a: $\Sigma_1^{(n)}(t)$ for $t\in V^c$:}
We will show that $\Sigma_1^{(n)}(t)$ converges to zero a.e. on $V^c$.
This is done by defining the measure
\begin{equation*}
	\theta(E) := \mu\big(E\cap V\big),\qquad E\in \mathcal B,
\end{equation*}
and 
\begin{equation*}
	F_r = \{t\in V^c : \limsup_n
		\|\Sigma_1^{(n)}(t)\|_X >1/r\} \subset V^c.
\end{equation*}
Observe that $\theta(F_r)=0$ and, by \eqref{eq:uniformestimate} and Theorem \ref{thm:boundgram}, 
\begin{equation*}
	\|\Sigma_1^{(n)}(t)\|_X \lesssim \sum_{i,j} \frac{q^{|i-j|}}{h_{ij}^{(n)}} 
	\theta(\supp N_i^{(n)}) N_j^{(n)}(t), \qquad t\in F_r,
\end{equation*}
which allows us to apply Lemma \ref{31-10-17-3} on $F_r$ and $\theta$ to get $\lambda(F_r)=0$ for all
$r>0$, i.e., $\Sigma_1^{(n)}(t)$ converges to zero a.e. on $V^c$.

\textsc{Part 1.b: $\Sigma_2^{(n)}(t)$ for $t\in V^c$:} Let  $\mathcal B_{V^c} = V^c \cap \mathcal B . $
Thus   $\mathcal B_{V^c}$ is the restriction of the Borel $\sigma$-algebra $\mathcal B $ to  $V^c . $  
In this case, we 
define the vector measure $\nu$ of bounded variation on $(V^c,\mathcal B_{V^c})$ by
\begin{equation*}
	\nu(A) := T(\charfun_A),\qquad A\in \mathcal B_{V^c}.
\end{equation*}
Here we use the second part of Lemma \ref{lem:charfun} to guarantee that the right 
hand side is defined and \eqref{eq:uniformestimate} ensures $|\nu|\leq \mu$.
Apply Lebesgue decomposition Theorem~\ref{thm:lebesgue} to get 
\begin{equation}\label{31-10-17-9}\dif\nu = g\dif\lambda + \dif\nu_s \end{equation}
where $g\in L^1_X$ and $|\nu_s|$ is singular to $\lambda$.
Observe that for all $f\in \cup S_n$, we have
\begin{equation}\label{eq:ext}
	\int f\dif\nu = T(f \charfun_{V^c}).
\end{equation}
Indeed, 
this holds for indicator functions by definition and each $f\in\cup S_n$
can be approximated in $L^1(\mu)$  by linear combinations of indicator 
functions. Therefore, \eqref{eq:ext} is established, since both sides of
\eqref{eq:ext} are continuous in $L^1(\mu)$.
So, 
\begin{align*}
	\Sigma_2^{(n)}(t) &= \sum_{i} \int N_i^{(n)}\dif\nu
	\cdot N_i^{(n)*}(t) \\
	&= \sum_{i} \int N_i^{(n)} g\dif \lambda \cdot
	N_i^{(n)*}(t) + \sum_{i} \int N_i^{(n)}\dif\nu_s \cdot
	N_i^{(n)*}(t).
\end{align*}
The first part is $P_n g$ for an $L^1_X$ function $g$ and this converges by
Theorem \ref{thm:aeconv} a.e. to $ g$.

To treat the second part $P_n\nu_s$, let $A\in \mathcal B_{V^c}$ be a subset of $V^c$ with the property
$\lambda(V^c\setminus A)=|\nu_s|(A)=0$, which is possible since $|\nu_s|$ is
singular to $\lambda$. For $x^*\in X^*$, we define the set 
\begin{equation*}
F_{r,x^*} := 
\{ t\in A: \limsup_n |(x^*P_n\nu_s)(t)| >1/r\}.
\end{equation*}
Since by  Theorem \ref{thm:boundgram}
\begin{align*}
|x^* P_n\nu_s(t)| &=\Big|\sum_{i,j} a_{ij}^{(n)}\int N_i^{(n)}\dif(x^*\circ\nu_s)\cdot N_j^{(n)}(t)\Big| \\
&\lesssim \sum_{i,j} \frac{q^{|i-j|}}{h_{ij}^{(n)}} |x^*\circ \nu_s|(\supp
N_i^{(n)}) \cdot N_j^{(n)}(t),
\end{align*}
we can apply Lemma \ref{31-10-17-3} to $F_{r,x^*}$ and the measure
$\theta(B)=|x^*\circ\nu_s|(B\cap V^c)$
 to  obtain $\lambda(F_{r,x^*})=0$.
Since the closure $X_0$ in $X$ of the set $\{P_n\nu_s(t) : t\in [0,1], n\in\mathbb
N\}$ is a separable subspace of $X$, by Lemma \ref{lem:dual}, there exists a sequence $(x^*_n)$ of
elements in $X^*$ such that for all $x\in X_0$ we have $\|x\| = \sup_n
|x^*_n(x)|$.
This means that we can write
\begin{equation*}
	F := \{ t\in A: \limsup_n \|P_n\nu_s(t)\| >0\} =\bigcup_{n,\ell=1}^\infty F_{\ell, x_n^*},
\end{equation*}
and thus, $\lambda(F)=0$, which shows that $P_n\nu_s$ tends to zero almost
everywhere on $V^c$ with respect to Lebesgue measure.

\textsc{Part 2: $t\in V$:}

Now, we consider $t\in V$
or more precisely $t\in U$. This makes no difference for considering a.e. convergence since 
the difference between $V$ and $U$ is a Lebesgue zero set.
We choose the index $j_0$ such that $t\in U_{j_0}$ and based on
the location of $t$, we decompose (using Lemma \ref{lem:charfun})
\begin{align*}
	g_n(t) &= P_n T(t) = \sum_{i} T (N_i^{(n)}) N_i^{(n)*}(t)
	\\
	&= \sum_{i} T (N_i^{(n)} \charfun_{V_{j_0}})\cdot N_i^{(n)*}(t)
	+  \sum_{i} T (N_i^{(n)} \charfun_{V_{j_0}^c})\cdot
	N_i^{(n)*}(t) \\
	&=: \Sigma_1^{(n)}(t) + \Sigma_2^{(n)}(t).
\end{align*}

\textsc{Part 2.a: $\Sigma_1^{(n)}(t)$ for $t\in U_{j_0}$:}

We now consider
\begin{equation*}
	\Sigma_1^{(n)}= \sum_{i}  T(N_i^{(n)} \charfun_{V_{j_0}})
	N_i^{(n)*}(t),\qquad t\in U_{j_0},
\end{equation*}
and perform the construction of the B-splines $(\bar N_j)$ and their dual
functions $(\bar N_j^*)$ corresponding to
$V_{j_0}$ described in Section \ref{sec:splineConstruction}.
Define the function
\begin{equation}
	\label{eq:defg}
	u(t) := \sum_j T(\bar N_j) \bar N_j^*(t),\qquad t\in U_{j_0},
\end{equation}
and first note that $\bar N_j\in W$ since by Lemma \ref{lem:convBspline}
it is the uniform limit of the functions $(N_j^{(n)}\charfun_{V_{j_0}})$, which,
in turn, are contained in $W$ by Lemma \ref{lem:charfun}. Therefore, $T(\bar
N_j)$ is defined. Moreover, the series in \eqref{eq:defg} converges pointwise
for $t\in U_{j_0}$, since $\lambda(\bar I(t))>0$, the sequence $j\mapsto \bar
N_j^*(t)$ admits a geometric decay estimate by \eqref{eq:barNDual} and the
inequality $\|T(\bar N_i)\|_X \leq \mu(\supp \bar N_i)$. If one additionally
notices that \eqref{eq:barNDual} implies the estimate $\|\bar
N_j^*\|_{L^1}\lesssim 1$ we see that the convergence in \eqref{eq:defg} takes
place in $L^1_X$ as well.
This implies $\langle u,\bar N_i\rangle = T(\bar N_i)$ for all $i$ by Lemma
\ref{lem:splineDual}.

Next, we show that if for all $n$, $(a_i)$ and $(a_i^{(n)})$ are sequences in $X$ so that 
for all $i$ we have $\lim_{n} a_i^{(n)} = a_i$,  and $\sup_i \|a_i\|_X +
\sup_{i,n}
\|a_i^{(n)}\|_X\lesssim 1$ it follows that 
\begin{equation}\label{eq:convDualSeries}
	\lim_n \sum_i (a_i^{(n)}-a_i) N_i^{(n)*}(t) = 0, \qquad t\in U_{j_0}.
\end{equation}
Indeed, let $\varepsilon>0$, the integer $L$ such that $q^L \leq
\varepsilon\cdot\inf_n
\lambda(I_n(t)) $ and $M$ sufficiently large that for all $n\geq
M$ and all $i$ with $|i-i_n(t)|\leq L$, we have $\|a_i^{(n)} - a_i \|_X \leq
\varepsilon\cdot\inf_n\lambda(I_n(t))$. Then, by Theorem \ref{thm:boundgram},
\begin{align*}
	\Big\| \sum_i (a_i^{(n)} - a_i) N_i^{(n)*}(t) \Big\|_X 
	&\leq \sum_i \| a_i^{(n)} - a_i \|_X \frac{q^{|i-i_n(t)|}}{\lambda(I_n(t))} \\
	&=\Big( \sum_{i: |i-i_n(t)|\leq L} + \sum_{i : |i-i_n(t)|>L}\Big)
	\| a_i^{(n)} - a_i \|_X \frac{q^{|i-i_n(t)|}}{\lambda(I_n(t))} \\
	&\lesssim \sum_{i:|i-i_n(t)|\leq L} \varepsilon q^{|i-i_n(t)|} +
	\sum_{i:|i-i_n(t)|>L} \frac{q^{|i-i_n(t)|}}{\lambda(I_n(t))} 
	\lesssim \varepsilon
\end{align*}

We now use these remarks to show that 
\begin{equation*}
	\lim_n \| \Sigma_1^{(n)}(t) - P_n u(t) \|_X = 0,\qquad t\in U_{j_0}.
\end{equation*}
Indeed, since $\langle u,\bar N_i\rangle = T(\bar N_i)$ for all $i$,
\begin{align*}
	\Sigma_1^{(n)}(t) - P_nu(t) &= \sum_i \big(T(N_i^{(n)}\charfun_{V_{j_0}})
	- \langle u,N_i^{(n)}\rangle\big) N_i^{(n)*}(t) \\
	&=\sum_i \big(T(N_i^{(n)}\charfun_{V_{j_0}})
	- T(\bar N_i)\big) N_i^{(n)*}(t)\\
	&\qquad + \sum_i \big(\langle u, \bar N_i\rangle
	- \langle u,N_i^{(n)}\rangle\big) N_i^{(n)*}(t).
\end{align*}
Now, observe that for all $i$, we have $T(N_i^{(n)}\charfun_{V_{j_0}})\to T(\bar
N_i)$ and $\langle u,N_i^{(n)}\rangle \to \langle u,\bar N_i\rangle$ since by
Lemma \ref{lem:convBspline}, $N_i^{(n)}$ converges uniformly to $\bar N_i$ on
$V_{j_0}$ and $u\in L^1$. Moreover all the expressions
$T(N_i^{(n)}\charfun_{V_{j_0}}), T(\bar N_i), \langle u,N_i^{(n)}\rangle$ are
bounded in $i$ and $n$. As a consequence, we can apply \eqref{eq:convDualSeries}
to both of the sums in the above display
to conclude 
\begin{equation*}
	\lim_n \| \Sigma_1^{(n)}(t) - P_n u(t) \|_X = 0,\qquad t\in U_{j_0}.
\end{equation*}
But we know that $P_n u(t)$ converges a.e. to $ u(t)$ by Theorem \ref{thm:aeconv}, this
means that also $\Sigma_1^{(n)}(t)$ converges to $u$  a.e.

\textsc{Part 2.b: $\Sigma_2^{(n)}(t)$ for $t\in U_{j_0}$:}
We show that $\Sigma_2^{(n)}(t)=\sum_{i} T (N_i^{(n)} \charfun_{V_{j_0}^c})\cdot
	N_i^{(n)*}(t)$ converges to zero for $t\in U_{j_0}$.
Let $\varepsilon>0$ 
and set $s= \inf_n \lambda(I_n(t))$, where we recall that $I_n(t)$ is the
grid interval in $\Delta_n$ that contains the point $t$. Since 
$s>0$ we can choose an open interval $O$ with the property $\mu(O\setminus
V_{j_0})\leq \varepsilon s$. Then, due to the fact that $t\in U_{j_0}$, we can
choose $M$ sufficiently large that both
intervals $(\inf O, t)$ and $(t,\sup O)$ contain 
$L$ points of the grid $\Delta_M$ where $L$ is such that $q^L
\leq \varepsilon s /\mu([0,1])$. Thus, we estimate for $n\geq M$ by
\eqref{eq:uniformestimate} and Theorem
\ref{thm:boundgram}
\begin{align*}
	\| \Sigma_2^{(n)}(t) \|_X & \leq \sum_{i} \mu(\supp N_i^{(n)}\cap
	V_{j_0}^c ) \frac{q^{|i-i_n(t)|}}{\lambda(I_n(t))} \\
	&\leq \frac{1}{s}\cdot \Big( \sum_{i: \supp N_i^{(n)}\cap O^c\neq \emptyset}
	+ \sum_{i:\supp N_i^{(n)}\subset O}\Big)\big( \mu(\supp N_i^{(n)}\cap
	V_{j_0}^c) q^{|i-i_n(t)|} \big) \\
	&\lesssim \frac{1}{s}\big( \mu([0,1]) q^L + \mu(O\setminus V_{j_0})
	\big) \lesssim \varepsilon.
\end{align*}
This proves that $\Sigma_2^{(n)}(t)$ converges to zero for $t\in U_{j_0}$.
\end{proof}

By looking at the above proof and employing the notation therein, we have
actually proved the following, explicit form of the Spline Convergence
Theorem:
\begin{thm}\label{3-11-17-1}
	Let $X$ be a Banach space with RNP and 
	$(g_n)$ be sequence in $L^1_X$ with the properties
	\begin{enumerate}
		\item $\sup_n \|g_n\|_{L^1_X} <\infty$, 
		\item $P_m g_n = g_m$ for all $m\leq n$.
	\end{enumerate}
Then, $g_n$ converges a.e. to the $L^1_X$-function
\begin{equation*}
 g\charfun_{V^c} + \sum_{j_0} \sum_j T(\bar N_{j_0,j}) \bar
	N_{j_0,j}^*\charfun_{U_{j_0}}.
\end{equation*}
Here, $g$  is defined by \eqref{31-10-17-9},
and  for each $j_0$, $(\bar N_{j_0,j})$ and $(\bar N_{j_0,j}^{*})$ are 
the B-splines and their dual functions 
constructed in Section
\ref{sec:splineConstruction} corresponding to $V_{j_0}$.
\end{thm}
\begin{rem} In order to emphasize the pivotal role of the set $V$ and its complement
we note that the proof  of Theorem~\ref{3-11-17-1} implies the following: 
If $(g_n)$ be sequence in $L^1_X$ such that 
	\begin{enumerate}
		\item $\sup_n \|g_n\|_{L^1_X} <\infty$, 
		\item $P_m g_n = g_m$ for all $m\leq n$
	\end{enumerate}
and if $\lambda(V^c) = 0$ then,  {\em without any} condition on the  Banach space $X$, 
 $g_n$ converges a.e. to  $$\sum_{j_0} \sum_j T(\bar N_{j_0,j}) \bar
	N_{j_0,j}^*\charfun_{U_{j_0}}.$$
\end{rem}

\begin{rem}
Based on the results of the present paper, 
an intrinsic spline characterization of the Radon-Nikod\'{y}m property 
in terms of splines was obtained by the second named author in \cite{Passenbrunner2019}.
The result in \cite{Passenbrunner2019} establishes the full analogy between spline and
 martingale convergence.
\end{rem}

\subsection*{Acknowledgments}
We are grateful to R. Lechner (Linz) for suggesting  that spline convergence  be 
studied in the context of Banach spaces with the RNP.
The research of M. P. was supported by the FWF-projects {Nr.}P27723 and {Nr.}P32342.
The research of P. F. X. M. was supported by  FWF-project {Nr.}P28352.
\bibliographystyle{plain}
\bibliography{convergence}
\end{document}